\documentclass{amsart}

\usepackage{amsmath,amsfonts,amssymb,amsthm}
\usepackage{geometry} 
\usepackage{color} 
\usepackage{xy} 
\usepackage[backref=page]{hyperref} 
\usepackage[final]{showlabels} 
\usepackage{enumitem} 

\newcommand{\p}{\partial}

\newcommand{\m}{\mathfrak{m}}
\newcommand{\g}{\mathfrak{g}}
\newcommand{\h}{\mathfrak{h}}
\newcommand{\mk}{\mathfrak{k}}

\newcommand{\s}{\mathfrak{sl}}
\newcommand{\gl}{\mathfrak{gl}}
\newcommand{\so}{\mathfrak{so}}
\newcommand{\mo}{\mathfrak{o}}
\newcommand{\spl}{\mathfrak{sp}}
\newcommand{\pgl}{\mathfrak{pgl}}
\newcommand{\psl}{\mathfrak{psl}}

\newcommand{\rad}{\mathfrak{rad}}

\newcommand{\ssimp}{\mathrm{ss}}

\newcommand{\bI}{\mathbb{I}}
\newcommand{\bV}{\mathbb{V}}

\newcommand{\bZ}{\mathbb{Z}}

\newcommand{\bN}{\mathbb{N}}

\newcommand{\moM}{\mo(M)}
\newcommand{\omoM}{\overline{\mo}(M)}
\newcommand{\moN}{\mo(N)}
\newcommand{\moI}{\mo(\mathbb{I}_{n})}
\newcommand{\moD}{\mo(D)}
\newcommand{\HomV}{\mathrm{Hom}\left(\bV_{\moN},\bV_{\gl_{m}}\right)}
\newcommand{\HomkV}{\mathrm{Hom}\left(\bV_{\mk},\bV_{\h}\right)}

\theoremstyle{plain}
\newtheorem{thm}{Theorem}[section]

\newtheorem{prop}[thm]{Proposition}

\theoremstyle{definition}

\theoremstyle{remark}
\newtheorem{remi}[thm]{Remark}

\theoremstyle{plain}
\newtheorem{mthm}{Theorem}
\newtheorem{mprop}[mthm]{Proposition}

\begin{document}

\title{The Lie algebra preserving a degenerate bilinear form}
\author{James Waldron}
\date{\today}

\begin{abstract}
Let $k$ be an arbitrary field and $d$ a positive integer.
For each degenerate symmetric or antisymmetric bilinear form $M$ on $k^{d}$ we determine the structure of the Lie algebra of matrices that preserve $M$, and of the Lie algebra of matrices that preserve the subspace spanned by $M$.
We show that these Lie algebras are semidirect products of classical Lie algebras and certain representations, and determine their radicals, derived series and semisimple quotients.
Our main motivation and application is to determine the structure of the graded Lie algebra of derivations of each commutative or graded commutative algebra with Hilbert polynomial $1+dt+t^{2}$.
Some of our results apply to more general bilinear forms and graded algebras. 
\end{abstract}

\maketitle


\section{Introduction}
\label{sec: intro}
Let $k$ be a field and $d$ a positive integer.
To each $d \times d$ matrix $M$ we can associate the following Lie subalgebras of the Lie algebra $\gl_{d}$ of $d\times d$ matrices:
\begin{align*}
\moM & :=\left\{X \in \gl_{d} \; | \; X^{t}M+MX=0\right\} \\
\omoM &  :=\left\{X \in \gl_{d} \; | \; X^{t}M+MX \in kM\right\}
\end{align*}
where $X^t$ denotes the transpose of $X$ and $kM$ the linear span of $M$.
The formula $X\cdot M = X^{t}M + MX$ defines an action of $\gl_{d}$ on the space of bilinear forms on $k^{d}$, where the matrix $M$ corresponds to the bilinear form $\left(v,w\right) = v^{t}Mw$.  
From this perspective $\moM$ is the stabilizer in $\gl_{d}$ of $M$, and $\omoM$ is the stabilizer in $\gl_{d}$ of the subspace spanned by $M$.

For example, if $M=0$ then $\omoM=\moM=\gl_{d}$.
At the other extreme, if $M=\bI_{d}$ is the identity matrix then $\moM$ is the Lie algebra of antisymmetric matrices and isomorphic to $\so_{d}$ if $\mathrm{char}\,k\ne 2$.
If $d=2n$ is even and $M$ is equal to the matrix 
\begin{equation*}
\Pi_{2n} := 
\begin{pmatrix}
0&\bI_{n}\\
-\bI_{n}&0
\end{pmatrix}
\end{equation*}	
then $\moM = \spl_{2n}$ is the symplectic Lie algebra of size $2n$. 
If every element in $k$ is a square and $M$ is nondegenerate and either symmetric or antisymmetric then $\moM$ is isomorphic to one of the latter two examples, and except for some low dimensional exceptions either $\moM$, or if $\mathrm{char}\,k=2$ a quotient of a derived subalgebra of $\moM$ is simple, see Appendix \ref{sec: bg} and references therein.

In this paper we study in detail the structure of the Lie algebras $\moM$ and $\omoM$ when $M$ is a \emph{singular} symmetric or antisymmetric matrix.
In this case $M$ corresponds to a \emph{degenerate} symmetric or antisymmetric bilinear form.
We show that these Lie algebras are semidirect products of classical Lie algebras and certain representations, determine their radicals, semisimple quotients and derived series, and determine the relation between $\omoM$ and $\mo{\left(M\right)}$.
Some of our results apply to more general $M$, in particular to matrices which are congruent to the block diagonal sum of a zero matrix and a nonsingular matrix.

Our main motivation and application and is to determine the structure of the Lie algebra of derivations of each commutative or graded commutative graded $k$-algebra $A$ with Hilbert polynomial $1+dt+t^{2}$.
The commutative $k$-algebras that admit such a grading coincide with the local $k$-algebras with residue field $k$ and maximal ideal $\m$ satisfying $\mathrm{dim}_{k} \, \m/\m^{2} = d$, $\mathrm{dim}_{k} \, \m^{2} = 1$ and $\m^{3}=0$.
These algebras appear in the classification of low dimensional commutative algebras completed by Poonen \cite{Poonen08} as algebras of type $\vec{d}=\left(d,1\right)$ in the notation of loc.\ cit.\
As explained in \cite[Ex.\ 1.5]{Poonen08} and recalled in \S \ref{sec: mainder} below, the multiplication in an algebra of this type is determined by a $d \times d$ square matrix $M$. 
We show that $\omoM$ corresponds to the Lie algebra of degree zero graded derivations of $A$ and $\moM$ to the degree zero graded derivations that act trivially on $\m^{3}$.
Apart from two exceptions related to the Witt algebras in characteristic 2 and 3 (see Remark \ref{rem: Witt} for the definition), the full Lie algebra of derivations is a semidirect product of $\omoM$ and a certain representation.

\subsection{Notation}
\label{sec: set-up}
Before stating our main results we fix some notation.
We continue to use the notation $k$, $d$, $\moM$, $\omoM$, $\bI_{n}$ and $\Pi_{2n}$ established in the introduction. 
The letters $m,n$ and $i$ always denote integers, with $m,n \ge 0$. 

We denote the vector space of $m\times n$ matrices over $k$ by $k^{m\times n}$, the span of a matrix $M \in k^{m\times n}$ by $kM$ or $k(M)$, and for $m=n$ the $m\times m$ zero matrix by $0_{m}$.

Two matrices $M,M' \in k^{d\times d}$ are called congruent if there exists an invertible matrix $g \in \mathrm{GL}_{d}\left(k\right)$ such that $g^tMg=M'$, in which case $\mo{\left(M\right)} \cong \mo\left(M'\right)$ and $\overline{\mo}{\left(M\right)} \cong \overline{\mo}\left(M'\right)$ via $X\mapsto g^{-1}Xg$. 
If $M\in k^{d\times d}$ is either symmetric or antisymmetric then $M$ is congruent to a block-diagonal matrix of exactly one of the forms
\begin{align*}
0_{m}\oplus \bI_{n} & :=
\begin{pmatrix}
0_{m} & 0 \\
0 & D
\end{pmatrix} 
\;\;,\;\; d=m+n \\
0_{m}\oplus \Pi_{2n} & :=
\begin{pmatrix}
0_{m} & 0 \\
0 & \Pi_{2n}
\end{pmatrix} 
\;\;,\;\; d=m+2n
\end{align*}
where $D \in k^{n\times n}$ is a nonsingular diagonal matrix \cite{Albert38,KaplanskyBook03}.
If the determinant of $D$ is a square in $k$ (e.g.\ if $k$ is algebraically closed) then $D$ is congruent to $\bI_{n}$, and if moreover $\mathrm{char}\,k \ne 2$ then $\moD \cong \so_{n}$.
(Over a field of $\mathrm{char} \ne 2$ the Lie algebra $\so_{n}$ is often defined using a non-diagonal matrix that is congruent to $\bI_{n}$ \cite[\S I]{HumphreysBook72}.)
This condition on $D$ is equivalent to the condition that the discriminant of the corresponding bilinear form is the identity element in the square class group $k^{\times} / k^{\times 2}$.

See \cite{HumphreysBook72} for generalities on Lie algebras.
We denote the centre of a Lie algebra $\g$ by $Z(\g)$, the radical by $\rad(\g)$, the maximal semisimple quotient $\g/\rad(\g)$ by $\g^{\ssimp}$ and the $i$'th derived subalgebra by $\g^{(i)}$, $i\ge 1$.

If $V$ is a representation of a Lie algebra $\g$ then we denote the action of $x\in \g$ on $v \in V$ by $x\cdot v$.
We denote by $\g V$ the subrepresentation spanned by $\left\{x\cdot v \; | \; x \in \g,v\in V\right\}$, and by $\g \ltimes V$ the semidirect product Lie algebra where $V$ is considered as an abelian Lie algebra. 
 
If $V$ resp.\ $W$ is a representation of the Lie algebra $\g$ resp.\ $\h$ then we consider the vector space $\mathrm{Hom}\left(W,V\right)$ as a representation of $\g\oplus \h$ with action $\left(\left(x,y\right)\cdot \phi\right) \left(w\right) = x\cdot \phi\left(v\right) -\phi\left(y\cdot w\right)$.
If $\g=\h$ then we also consider $\mathrm{Hom}\left(V,W\right)$ as a representation of $\g$ with action given by $\left(x\cdot \phi\right)\left(v\right) = x\cdot\phi\left(v\right)-\phi\left(x\cdot v\right)$.

We denote the restriction of the vector representation of $\gl_{n}$ to a subalgebra $\g \subseteq \gl_{n}$ by $\bV_{\g}$.
If $\h \subsetneq \g \subseteq \gl_{n}$ are subalgebras then we occasionally denote $\bV_{\h}$ by $\bV_{\g}$.
For $\lambda \in k$ we denote by $k_{\lambda}$ the representation of $\gl_{1}$ on $k$ where $1\in \gl_{1}$ acts by multiplication by $\lambda$. 
\newpage
\subsection{The Lie algebras $\moM$ and $\omoM$}
\label{sec: maingen}
%
\begin{mthm}
\label{thm: table}
For $d=m+n$ and $M=0_{m}\oplus D$ with $D$ a nonsingular diagonal matrix, or $d=m+2n$ and $M=0_{m}\oplus\Pi_{2n}$, the structure of the Lie algebras $\moM$ and $\omoM$ is given in the following table:
\[
	\begin{array}{|c|c|c|c|c|}
	\hline
	M
	&\mathrm{char}\,k
	&\moM
	&\text{\emph{Isomorphism type of }} \moM 
	&\omoM 
	\\ \hline
	\begin{pmatrix}
	0_{m}&0\\
	0&D
	\end{pmatrix}
	&2
	&\begin{pmatrix}
	\gl_{m}& k^{m\times n} \\
	0& \moD 
	\end{pmatrix}
	&\left(\gl_{m}\oplus \moD\right) \ltimes \mathrm{Hom}\left(\bV_{\moD} , \bV_{\gl_{n}}\right)
	& \moM
	\\ 
	&\ne 2
	& ''
	& ''
	& \moM \oplus k\bI_{m+n}
	\\ \hline
	\begin{pmatrix}
	0_{m}&0\\
	0&\Pi_{2n}
	\end{pmatrix}
	&2
	& \begin{pmatrix}
	\mathfrak{gl}_{m} & k^{m\times 2n}\\
	0 & \spl_{2n} 
	\end{pmatrix}
	&\left(\gl_{m}\oplus \spl_{2n}\right) \ltimes \mathrm{Hom}\left(\bV_{\spl_{2n}} , \bV_{\gl_{n}}\right)
	& \moM \rtimes k  \, (0_{m+n}\oplus \bI_{n})
	\\ 
	&\ne2
	& ''
	& ''
	& \moM \oplus k\bI_{m+2n}
	\\ \hline
	\end{array}
\]
\end{mthm}
\begin{remi}
Theorem \ref{thm: table} provides examples of non-congruent matrices $M$ and $M'$ with $\moM \cong \mo\left(M'\right)$.
For example, if $\mathrm{char}\,k\ne 2$, $d=3$, $M=0_{2} \oplus \bI_{1}$ and $M'=0_{1} \oplus \Pi_{2}$, then there is an isomorphism
\[
\gl_{2} \ltimes \bV_{\gl_{2}}
\cong \left(\gl_{1} \oplus \spl_{2}\right) \ltimes \mathrm{Hom}\left(\bV_{\spl_{2}},\bV_{\gl_{1}}\right)
\]
arising from the isomorphisms $\gl_{2} \cong \gl_{1} \oplus \s_{2}$ and $\s_{2} \cong \spl_{2}$. 
\end{remi}
Theorem \ref{thm: table} is an application of the following more general result.
\begin{mthm}
\label{thm: gen}
Suppose that $d=m+n$, $N$ is a nonsingular $n\times n$ matrix, and $M=0_{m}\oplus N$. 
\begin{enumerate}
\item \label{gen: mat} $\moM$ consists of the matrices of the form 
\[
	\begin{pmatrix}
	\gl_{m} & k^{m\times n} \\
	0 & \moN
	\end{pmatrix}.
\]
\item \label{gen: isom} $\moM \cong \left(\gl_{m} \oplus \moN\right) \ltimes \HomV$.
\item \label{gen: rad} $\rad(\moM) \cong \left(\rad\left(\gl_{m}\right) \oplus \rad\left(\moN\right)  \right)  \ltimes \HomV$.
\item \label{gen: ss} $\moM^{\ssimp} \cong \left(\gl_{m}\right)^{\ssimp} \oplus \moN^{\ssimp}$.
\item \label{gen: der} $\moM^{\left(1\right)} \cong \left(\s_{m} \oplus \moN^{\left(1\right)}\right) \ltimes \mathrm{Hom}\left(\bV_{\moN^{\left(1\right)}} , \bV_{\s_{m}}\right)$.
\item \label{gen: der2} Let $i\ge 2$. Assume that one of the following conditions holds:
\begin{enumerate} 
	\item[(i)] $\mathrm{char} \, k = 2$ and $m\ge 3$. 
	\item[(ii)] $\mathrm{char} \, k \ne 2$ and $m \ge 2$. 
	\item[(iii)] $\moN^{\left(i-1\right)} \bV_{\moN} = \bV_{\moN}$. 
\end{enumerate}
Then $\moM^{\left(i\right)} \cong \left(\s_{m}^{\left(i-1\right)} \oplus \moN^{\left(i\right)}\right) \ltimes \mathrm{Hom}\left(\bV_{\moN^{\left(i\right)}} , \bV_{\s_{m}^{\left(i-1\right)}}\right)$.
\end{enumerate}
\end{mthm}
The following Proposition describes the relation between $\moM$ and $\omoM$ for general $M$.
\begin{mprop}
\label{prop: barses}
Assume that $M\ne 0$.
The following statements hold.
\begin{enumerate}
\item If $\mathrm{char}\,k\ne 2$ then $\omoM = \moM \oplus k\bI_{n}$ and $\omoM^{\left(1\right)} = \moM^{\left(1\right)}$.
\item \label{barses2} If $\mathrm{char}\,k=2$ then $\moM$ is either equal to $\omoM$ or there is a short exact sequence of Lie algebras
\begin{equation}
\label{eqn: ses}
0 \to \moM \to \omoM \to k \to 0
\end{equation}
where $\omoM \to k$ maps $X$ to the unique $\lambda \in k$ satisfying $X^tM + MX = \lambda M$.
In the latter case $\omoM$ is isomorphic to a semidirect product $\moM \rtimes \gl_1$ and $\moM$ is a codimension one ideal containing $\omoM^{\left(1\right)}$.  
\end{enumerate}
\end{mprop}
\begin{remi}
\label{rem: bothcases}
Suppose that $\mathrm{char}\, k = 2$.
We have borrowed the notation $\omoM$ from the proof of \cite{KrutovL18}.
The nondegenerate cases $M=\mathbb{I}_{n}$ and $M=\Pi_{2n}$ of Proposition \ref{prop: barses}\eqref{barses2} are stated in the proof of \cite[Thm.\ 5.1]{KrutovL18}, where $\omoM$ is related to the Lie algebra of derivations of a simple quotient of $\moM^{\left(2\right)}$.
It is noted in loc.\ cit.\ that both of the possibilities in Proposition \ref{prop: barses}\eqref{barses2} occur in the nondegenerate symmetric case: $\moI = \overline{\mo}(I_n)$ whereas $\overline{\mo}\left(\Pi_{2n}\right) = \spl_{2n} \rtimes k\left(0_{n}\oplus\bI_n\right)$ is a non-trivial semidirect product. 

We also note that a semidirect product similar to the right hand side of Theorem \ref{thm: gen}\eqref{gen: isom} appears in \cite[\S 3.2]{Lebedev06} in the description of Lie algebras associated to certain equivalence classes of bilinear forms.
\end{remi}
\begin{remi}
\label{rem: bil}
Our results can be applied to bilinear forms as follows.
Suppose that $B:V\times V \to k$ is a bilinear form on a finite dimensional $k$-vector space $V$.
One can define Lie subalgebras $\mo\left(B\right)$ and $\overline{\mo}\left(B\right)$ of $\gl\left(V\right)$ in a natural way.
If one fixes an ordered basis of $V$ then the associated isomorphism $\gl\left(V\right) \cong \gl_{\mathrm{dim}\;V}$ restricts to isomorphisms $\mo\left(B\right) \cong \moM$ and $\overline{\mo}\left(B\right) \cong \omoM$ where $M$ is the Gram matrix of $B$.
The condition that $M$ is congruent to a matrix of the form $0_m \oplus N$ with $N$ nondegenerate is equivalent to the condition that the bilinear form $B_{V/V^{\perp}}$ induced on $V/V^{\perp}$ is nondegenerate, where
\[
	V^{\perp} = \left\{ x \in V \; | \; B\left(x,y\right) = B\left(y,x\right) = 0 \; \text{for all} \; y \in V \right\}.
\]
In this case the basis independent analogue of Theorem \ref{thm: gen}\eqref{gen: isom} is that 
\[
\mo\left(B\right) \cong \left(\gl\left(V^{\perp}\right) \oplus \mo\left(B_{V/V^{\perp}}\right) \right) \ltimes \mathrm{Hom}\left(V/V^{\perp},V^{\perp}\right). 
\] 
We note that $\mathrm{Hom}\left(V/V^{\perp},V^{\perp}\right)$ appears in the proof of \cite[Lem. 3.1]{AhmadiIS14} where it is denoted $I_{\mathrm{Rad}\left(V\right)}$.
\end{remi}
\subsection{Further calculations}
For $M$ a singular symmetric or antisymmetric matrix we calculate the radical, semisimple quotient and derived series of the Lie algebra $\moM$.
These statements and their proofs are contained in \S \ref{sec: struc}.
\subsection{Derivations of graded algebras}
\label{sec: mainder}
Suppose that $A$ is a finite dimensional unital graded $k$-algebra with Hilbert polynomial $1 + dt + t^2$ so that
\[
A = k1_{A} \oplus A_{\left(1\right)} \oplus A_{\left(2\right)}
\] 
with $\mathrm{dim}\;A_{\left(1\right)} = d$ and $\mathrm{dim}\;A_{\left(2\right)} = 1$. 
Assume that $A$ is generated by $A_{\left(1\right)}$.
It follows from the grading that the multiplication in $A$ is equivalent to a non-zero $A_{\left(2\right)}$-valued bilinear form on $A_{\left(1\right)}$.
More concretely, if we fix an ordered basis $e_1,\dots,e_d$ of $A_{\left(1\right)}$ and a basis $e$ of $A_{\left(2\right)}$ then the multiplication in $A$ is determined by the $d\times d$ matrix $M$ defined by $e_i e_j = M_{ij}e$.
Changing the basis of $A_{\left(1\right)}$ resp.\ $A_{\left(2\right)}$ corresponds to replacing $M$ by a congruent matrix resp.\ rescaling $M$ by a non-zero element of $k$.
Therefore, if every element in $k$ is a square (e.g.\ if $k$ is algebraically closed) then isomorphism classes of algebras of this type correspond to congruence classes of non-zero square matrices of size $d$.
Note that $A$ is commutative resp.\ graded commutative if and only if $M$ is symmetric resp.\ antisymmetric.

The commutative algebras of this type and their description in terms of symmetric matrices appear in the classification of low dimensional commutative algebras in work of Poonen \cite{Poonen08}.
These algebras coincide with the local $k$-algebras with residue field $k$ and maximal ideal $\m$ satisfying $\mathrm{dim}_{k}\; \m/\m^{2} = d$, $\mathrm{dim}_{k} \; \m^{2} = 1$ and $\m^{3}=0$, which are the algebras of type $\vec{d}=(d,1)$ in the notation of loc.\ cit.

Let $A$ be an algebra of the above type.
The Lie algebra of derivations of $A$
	\[
	\g := \left\{\p \in \gl(A) \; | \; \p(ab) = \p(a)b+a\p(b)\; \text{for all} \; a,b \in A\right\}
	\]
inherits a grading from $A$ defined by
	\[
	\g_{(i)} := \left\{\p \in \g \; | \; \p(A_{(j)}) \subseteq A_{(j+i)} \; \text{for all} \; 0\le j \le 2 \right\}
	\]
and satisfying $\left[\g_{\left(i\right)},\g_{\left(j\right)}\right] \subseteq \g_{\left(i+j\right)}$ for any $i,j \in \bZ$. 
As a derivation $\p \in \g$ is determined by its restriction to the generating set $A_{\left(1\right)}$ the subspace $\g_{\left(i\right)} \subseteq \g$ can be identified with a subspace of $\mathrm{Hom}\left(A_{\left(1\right)},A_{\left(i\right)}\right)$.
In particular, $\g_{\left(i\right)}=0$ for $|i|\ge 2$ and 
	\[
	\g = \g_{\left(-1\right)} \oplus \g_{\left(0\right)} \oplus \g_{\left(1\right)}.
	\]
The natural representation of $\g$ on $A$ restricts to a representation of $\g_{\left(0\right)}$ on each $A_{\left(i\right)}$ which we denote by $\pi:\g_{\left(0\right)} \to \gl\left(A_{\left(2\right)}\right)$ in the case $i=2$.
The Lie algebra $\g_{\left(0\right)}$ acts on each $\g_{\left(i\right)}$ via the restriction of the adjoint representation.
With respect to the identification above $\g_{\left(i\right)}$ is naturally identified as a $\g_{(0)}$ representation with a subrepresentation of $\mathrm{Hom}\left(A_{\left(1\right)},A_{\left(i\right)}\right)$.
See Remark \ref{rem: Witt} below for the definition of the Witt algebras $W(n)$ appearing in the following result.
\begin{mthm}
\label{thm: der}
Let $M$ be the matrix corresponding to some choice of bases of $A_{\left(1\right)}$ and $A_{\left(2\right)}$. 
\begin{enumerate}
\item \label{der: g0} $\g_{\left(0\right)} \cong \omoM$.
\item $\mathrm{Ker} \, \pi \cong \moM$. 
\item \label{der: g1ab} The Lie algebra $\g_{\left(1\right)}$ is abelian.
\item \label{der: g0g1} $\g_{\left(1\right)} \cong \mathrm{Hom}\left(A_{\left(1\right)},A_{\left(2\right)}\right)$.
\end{enumerate}
For the remaining statements assume that $A$ is either commutative or graded commutative. 
\begin{enumerate}
\setcounter{enumi}{4}
\item \label{der: sl2} If $\mathrm{char}\,k=3$ and $d=1$ then $\g_{\left(-1\right)} \cong \mathrm{Hom}\left(A_{\left(1\right)},A_{\left(0\right)}\right)$, $A \cong k\left[t\right]/\left(t^{3}\right)$ and $\g$ is isomorphic to the Witt algebra $W(1)\cong\s_{2}$. 
\item If $\mathrm{char}\,k=2$, $d=2$ and $M$ is congruent to $\Pi_{2}$ then $\g_{\left(-1\right)} \cong \mathrm{Hom}\left(A_{\left(1\right)},A_{\left(0\right)}\right)$, $A\cong k\left[t_{1},t_{2}\right]/\left(t_{1}^{2},t_{2}^{2}\right)$ and $\g$ is isomorphic to the Witt algebra $W(2)$.
\item \label{der: other} In all other cases the following statements hold:
\begin{enumerate}[label=(\roman*)]
	\item \label{other: g(-1)} $\g_{\left(-1\right)} = 0$,
	\item \label{other: isom} $\g \cong \g_{\left(0\right)} \ltimes \mathrm{Hom}\left(A_{\left(1\right)},A_{\left(2\right)}\right)$,
\
	\item \label{other: rad} $\rad\left(\g\right) \cong \rad\left(\omoM\right) \ltimes \mathrm{Hom}\left(A_{\left(1\right)},A_{\left(2\right)}\right)$,
	\item \label{other: ss} $\g^{\ssimp} \cong \omoM^{\ssimp}$.
	\item \label{other: iterated} $\g$ is isomorphic to an iterated semi-direct product, and in particular:
	\[
	\g \cong 
	\begin{cases}
	( (\gl_{m} \oplus \moI) \ltimes k^{m \times n}) \ltimes k^{d} & \text{ if } \; M \sim 0_{m} \oplus \bI_{n} \; \text{ and }  \; \mathrm{char} \, k = 2, \\
	( (\gl_{m} \oplus \so_{n} \oplus \gl_{1}) \ltimes k^{m \times n}) \ltimes k^{d} & \text{ if } \; M \sim 0_{m} \oplus \bI_{n} \; \text{ and }  \; \mathrm{char} \, k \ne 2, \\ 
	( (\gl_{m} \oplus \overline{\mo}(\Pi_{2n}) ) \ltimes k^{m \times 2n}) \ltimes k^{d} & \text{ if } \; M \sim 0_{m} \oplus \Pi_{2n} \; \text{ and }  \; \mathrm{char} \, k = 2,  \\
	( (\gl_{m} \oplus \spl_{2n} \oplus \gl_{1}) \ltimes k^{m \times 2n}) \ltimes k^{d} & \text{ if } \; M \sim 0_{m} \oplus \Pi_{2n} \; \text{ and }  \; \mathrm{char} \, k \ne 2,
	\end{cases}
	\]
	where `$\sim$' denotes the relation of congruence.
\end{enumerate}

\end{enumerate}
\end{mthm}
\begin{remi}
\label{rem: Witt}
For a field $k$ of characteristic $p>0$ and $n\in \bN$ the \emph{Witt algebra} $W(n)$ is by definition the Lie algebra of derivations of the commutative algebra $k\left[t_{1},\dots,t_{n}\right]/\left(t_{1}^{p},\dots,t_{1}^{p}\right)$.
The Lie algebra $W(n)$ is simple unless $p=2$ and $n=1$.
Except for some exceptional isomorphisms for small $p$ and $n$ the Witt algebras are not isomorphic to any simple Lie algebra of classical type and appear in the classification of finite dimensional simple Lie algebra over algebraically closed fields of characteristic $p\ge 5$.
See \cite{PremetS06,StradeBookI} for further details.
\end{remi}
\subsection{Relation to other work}
The square matrices $M$ such that $\moM$ is simple, semisimple or reductive (in the sense that all solvable ideals are central, though see \cite{MOhumphreys17} for a discussion of this notion when $\mathrm{char}\,k \ne 0$) were determined by Ahmadi, Izadi, Chaktoura, and Szechtman \cite{AhmadiCS14, AhmadiIS14}.
\subsection{Structure of the paper}
The proofs of Theorems \ref{thm: table} and \ref{thm: gen}, Proposition \ref{prop: barses} and Theorem \ref{thm: der} are contained the sections \ref{sec: proofthmgen} - \ref{sec: proofpropder}.
Note that Theorem \ref{thm: gen} is proved in \S \ref{sec: proofthmgen} before Theorem \ref{thm: table} in \S \ref{sec: proofthmtable}.
In \S \ref{sec: struc} we calculate the radicals, semisimple quotients and derived series of $\moM$ for $M$ symmetric or antisymmetric. 
In Appendix \ref{sec: bg} we collect a number of known facts about the structure of the Lie algebras $\gl_{n}$, $\s_{n}$, $\moD$, $\so_{n}$ and $\spl_{2n}$.
\subsection{Acknowledgements}
Some related computations of certain Lie algebras of derivations were done in forthcoming joint work with Joseph Dessi \cite{DessiW}.


\section{Proofs of main results}
%
\subsection{Preliminary results}
The following Propositions collects several statements about semidirect product Lie algebras and certain subalgebras of $\gl_{d}$ that will be used in the proof of Theorem \ref{thm: gen}.
%
\begin{prop}
\label{prop: sd}
Suppose that $V$ is a finite dimensional representation of a finite dimensional Lie algebra $\g$. 
\begin{enumerate}
\item \label{sd: rad} $\rad(\g \ltimes V) = \rad\left(\g\right) \ltimes V$. 
\item \label{sd: ss} $\left(\g\ltimes V\right)^{\ssimp} \cong \g^{\ssimp}$.
\item \label{sd: sol} $\g \ltimes V$ is solvable if and only if $\g$ is solvable.
\item \label{sd: der} $\left(\g \ltimes V\right)^{\left(1\right)} = \g^{\left(1\right)} \ltimes \g V$.
\item \label{sd: perf} If $\g$ is perfect and $\g V = V$ then $\g \ltimes \g V$ is perfect.
\end{enumerate}
\end{prop}
\begin{proof}
\begin{enumerate}[wide, labelwidth=!, labelindent=0pt]
\item The ideal $\rad(\g) \ltimes \g V$ is solvable because it is an extension of a solvable Lie algebra by an abelian Lie algebra.
It is equal to the radical of $\g\ltimes V$ because its image under the quotient morphism $\g \ltimes V \to \left(\g \ltimes V\right) / V \cong \g$ is equal to $\rad\left(\g\right)$ which is by definition the maximal solvable ideal in $\g$.
\item By \eqref{sd: rad} we have $\left(\g \ltimes V\right)^{\ssimp} = \left(\g \ltimes V\right) / \left(\rad\left(\g\right) \ltimes V\right) \cong \g / \rad(\g) = \g^{\ssimp}$. 
\item This follows immediately from \eqref{sd: ss}.
\item It follows from the definition of the semidirect product and the fact that $V$ is abelian that 
	\begin{align*}
	\left(\g \ltimes V\right)^{\left(1\right)} 
	& = \mathrm{span} \, \left\{ \left[ \left(x,v\right) , \left(y,w\right) \right]  \; | \; x,y \in \g, \; v,w \in V \right\} \\
	& = \mathrm{span} \, \left\{ \left[ \left(x,0\right) , \left(y,0\right) \right] , \left[ \left(z,0\right) , \left(0,v\right) \right] \; | \; x,y,z \in \g, \; v \in V \right\} \\
	& = \mathrm{span} \, \left\{ \left(x,0\right),\left(0,z \cdot v\right) \; | \; x \in \g^{\left(1\right)},z \in \g, \; v \in V \right\} \\
	& = \g^{\left(1\right)} \ltimes \g V.
	\end{align*}
as required.
\item This follows immediately from \eqref{sd: der}. 
\qedhere
\end{enumerate}
\end{proof}
\begin{prop}
\label{prop: block}
Suppose that $\h\subseteq \gl_{m}$ and $\mk\subseteq \gl_{n}$ are Lie subalgebras. Let $\g \subseteq \gl_{m+n}$ be subspace
	\[
	\begin{pmatrix}
	\h & k^{m\times n} \\
	0 & \mk
	\end{pmatrix}.
	\]
\begin{enumerate}
\item \label{block: subalg} $\g$ is a Lie subalgebra of $\gl_{m+n}$.
\item \label{block: isom} $\g \cong \left(\h\oplus \mk\right) \ltimes \HomkV$.
\item \label{block: rad} $\rad\left(\g\right) \cong \left(\rad\left(\g\right) \oplus \rad\left(\mk\right)\right) \ltimes \HomkV$.
\item \label{block: ssimp} $\g^{\ssimp} \cong \h^{\ssimp} \oplus \mk^{\ssimp}$.
\end{enumerate}
\end{prop}
\begin{proof}
First note that if $A,A'\in \h$, $D,D'\in \mk$ and $B,B' \in k^{m\times n}$ then
	\begin{align} 
	\label{eqn: blockdiag}
	\left[
	\begin{pmatrix}
	A & 0 \\
	0 & D 
	\end{pmatrix}
	,
	\begin{pmatrix}
	A' & 0 \\
	0 & D'
	\end{pmatrix}
	\right]
	& =
	\begin{pmatrix}
	\left[A,A'\right] & 0 \\
	0 & \left[D,D'\right]
	\end{pmatrix}
\end{align}
and
	\begin{align}	
	\label{eqn: blockaction}
	\left[
	\begin{pmatrix}
	A & B \\
	0 & D 
	\end{pmatrix}
	,
	\begin{pmatrix}
	0 & B' \\
	0 & 0
	\end{pmatrix}
	\right]
	& =
	\begin{pmatrix}
	0 & AB' -B'D \\
	0 & 0
	\end{pmatrix}.
	\end{align}
\begin{enumerate}[wide, labelwidth=!, labelindent=10pt]
\item Equations \eqref{eqn: blockdiag} and \eqref{eqn: blockaction} show that $\g$ is a subalgebra of $\gl_{m+n}$.
\item Equation \eqref{eqn: blockdiag} shows that the elements of $\g$ of the form
	\[
	\begin{pmatrix}
	A & 0 \\
	0 & D 
	\end{pmatrix}
	\]
constitute a subalgebra identified with $\h \oplus \mk$, and \eqref{eqn: blockaction} shows that those of the form
	\[
	\begin{pmatrix}
	0 & B' \\
	0 & 0
	\end{pmatrix}
	\]
constitute an abelian ideal identified with $k^{m\times n}$.
It follows that there is a canonical isomorphism $\g \cong \left(\h\oplus\mk\right) \ltimes k^{m\times n}$ where $A\oplus D$ acts on $B \in k^{m\times n}$ by $\left(A\oplus D\right) \cdot B = AB - BD$.
Under the canonical linear isomorphism $k^{m\times n} \cong \mathrm{Hom}\left(k^{n},k^{m}\right) = \HomkV$ this is exactly the action of $\h\oplus\mk$ on $\HomkV$, see \S \ref{sec: set-up}.
\item This follows from statement \eqref{block: isom}, Proposition \ref{prop: sd}\eqref{sd: rad}, and the fact that $\rad\left(\h\oplus \mk\right) = \rad\left(\h\right) \oplus \rad\left(\mk\right)$. 
\item This follows from statement \eqref{block: isom}, Proposition \ref{prop: sd}\eqref{sd: ss}, and the fact that $\left(\h\oplus\mk\right)^{\ssimp} \cong \h^{\ssimp}\oplus\mk^{\ssimp}$. 
\qedhere
\end{enumerate}
\end{proof}
\subsection{Proof of Theorem \ref{thm: gen}}
\label{sec: proofthmgen}
%
\begin{proof}
\begin{enumerate}[wide, labelwidth=!, labelindent=10pt]
\item Suppose that
	\begin{equation*}
	X = \begin{pmatrix}
	A & B \\
	C & D 
	\end{pmatrix}
	\end{equation*}
where $A,B,C$ and $D$ are matrices of sizes $m\times m$, $m\times n$, $n\times m$ and $n\times n$ respectively.
The equation $X^{t}M+MX=0$ is
	\[
	\begin{pmatrix}
	A & B \\
	C & D 
	\end{pmatrix}^{t}
	\begin{pmatrix}
	0 & 0 \\
	0 & N 
	\end{pmatrix}
	+
	\begin{pmatrix}
	0 & 0 \\
	0 & N 
	\end{pmatrix}
	\begin{pmatrix}
	A & B \\
	C & D 
	\end{pmatrix}
	=
	\begin{pmatrix}
	0 & 0 \\
	0 & 0
	\end{pmatrix}
	\]
or
	\[
	\begin{pmatrix}
	0 & C N \\
	NC & D^{t}N+ND 
	\end{pmatrix}
	= 
	\begin{pmatrix}
	0 & 0  \\
	0 & 0 
	\end{pmatrix}
	\]
which is equivalent to $D^{t}N+ND=0$ and $C=0$ because $N$ is nonsingular.
Therefore $X \in \moM$ if and only if $C=0$ and $D \in \moN$. 
\item This follows from (a) and Proposition \ref{prop: block}\eqref{block: isom}.
\item This follows from (a) and Proposition \ref{prop: block}\eqref{block: rad}.
\item This follows from (a) and Proposition \ref{prop: block}\eqref{block: ssimp}.
\item  We may assume that $m\ge 1$ because if $m=0$ then $\moM=\moN$ and there is nothing to prove.
Then $\gl_{m}^{\left(1\right)} = \s_{m}$ and
	\[ 
	\left(\gl_{m} \oplus \moN \right) \HomV = \HomV
	\] 
because if $\phi \in \HomV$ then $\phi = \left(\bI_{m}\oplus 0\right) \cdot \phi$.
The result then follows from Proposition \ref{prop: sd}\eqref{sd: der}. 
\item  We may assume that $m\ge 1$ because if $m=0$ then $\moM=\moN$ and there is nothing to prove.
It is sufficient to show that in each of the three cases (i),(ii) and (iii) that 
	\begin{equation}
	\label{eqn: derhom} 
	\left(\gl_{m}^{\left(i-1\right)} \oplus \moN^{\left(i-1\right)} \right) \HomV = \HomV
	\end{equation}
because then $\gl_{m}^{\left(j-1\right)} \oplus \moN^{\left(j-1\right)} \subseteq \gl_{m}^{\left(i-1\right)} \oplus \moN^{\left(i-1\right)}$ for all $1\le j \le i$, which implies that 
	\begin{equation*}
	\left(\gl_{m}^{\left(j-1\right)} \oplus \moN^{\left(j-1\right)} \right) \HomV = \HomV
	\end{equation*}
and the result follows by induction and Proposition \ref{prop: sd}\eqref{sd: der}.

(i)\&(ii).
In either of these cases $\s_{m}$ is perfect and so $\gl_{m}^{\left(i-1\right)} = \s_{m}^{\left(i-2\right)} = \s_{m}$.
It is easy to see that $\s_{m}\bV_{\gl_{n}} = \bV_{\gl_{n}}$ (in fact $\bV_{\gl_{n}}$ is simple).
As an $\s_{m}$ representation $\HomV$ is isomorphic to the direct sum of $n$ copies of $\bV_{\gl_{n}}$ and therefore $\s_{m} \, \HomV = \HomV$ and \eqref{eqn: derhom} holds.

(iii).
By definition $\moN$ preserves a nondegenerate bilinear form on $k^n$ and therefore $\bV_{\moN}^{*} \cong \bV_{\moN}$.
By assumption $\moN^{\left(i-1\right)} \bV_{\moN} = \bV_{\moN}$ and therefore $\moN^{\left(i-1\right)} \bV_{\moN}^{*} = \bV_{\moN}^{*}$ also.
As an $\moN^{\left(i-1\right)}$ representation $\HomV$ is isomorphic to the direct sum of $m$ copies of $\bV_{\moN}^{*}$ and therefore $\moN^{\left(i-1\right)} \HomV = \HomV$ and \eqref{eqn: derhom} holds.
\qedhere
\end{enumerate}
\end{proof}
\begin{remi}
For the remainder of the paper we will use the isomorphisms established in the proof of Theorem \ref{thm: gen}\eqref{gen: isom} without comment.
In particular, in the proofs of Propositions \ref{prop: OI2}-\ref{prop: OP} we prove the statements about radicals, semisimple quotients and derived series by proving the corresponding statements about $\left(\gl_m \oplus \moN\right) \ltimes \HomV$.
\end{remi}
\subsection{Proof of Proposition \ref{prop: barses}}
%
\begin{proof}
We first establish some facts for general $k$ which will prove (b), and then specialize to the case $\mathrm{char}\,k \ne 2$ to prove (a).

If $X \in \omoM$ then the scalar $\lambda \in k$ satisfying $X^{t}M+MX=\lambda M$ is uniquely determined by and depends linearly on $X$ because $k$ is a field and $M \ne 0$.
It follows that there is a well-defined linear map $\omoM \to k$ mapping $X \in \omoM$ to the corresponding $\lambda \in k$.
The kernel of this map is $\moN$.
Considering $k$ as an abelian Lie algebra, this map is a Lie algebra homomorphism if and only if its kernel $\moN$ contains $\moN^{\left(1\right)}$.
If $X^{t}M+MX=\lambda M$ and $Y^{t}M+MY=\mu M$ then
	\begin{align*}
	\left[X,Y\right]^{t}M + M\left[X,Y\right] & =
	Y^{t}X^{t}M - X^{t}Y^{t}M + MXY - MYX \\
	& = \left(Y^{t}X^{t}M + Y^{t}MX\right) + \left(X^{t}MY + MXY\right)  \\
	& \;\;\;\;\; - \left(Y^{t}MX + MYX\right) -\left(X^{t}Y^{t}M + X^{t}MY\right)   \\
	& = \lambda Y^{t}M + \lambda MY  -\mu MX -\mu X^{t}M   \\
	& = \lambda \mu M - \mu \lambda M \\
	& = 0
	\end{align*}
and therefore $\left[X,Y\right] \in \moM$ as required.
It follows that either $\moM = \omoM$, or the morphism $\omoM \to k$ is surjective, there is a short exact sequence 
	\[
	0 \to \moM \to \omoM \to k \to 0,
	\]
and $\moM^{\left(1\right)} \subseteq \moM$.
In the second case $\omoM = \moM \rtimes kX \cong \omoM \rtimes \gl_{1}$ where $X$ is any element of $\omoM \backslash \moM$.
This proves (b). 

(a) If $\mathrm{char}\,k\ne 2$ then $\bI_{n}M + M\bI_{n} = 2M$ and therefore $\bI_{n} \in \omoM \backslash \moM$.
As $\bI_{n}$ is central in $\gl_{n}$ we have $\omoM = \moM \rtimes k\bI_{n} = \moM \oplus k\bI_{n}$.
\end{proof}
%
\subsection{Proof of Theorem \ref{thm: table}}
\label{sec: proofthmtable}
%
\begin{proof}
The entries in the table in Theorem \ref{thm: table} all follow from Theorem \ref{thm: gen}\eqref{gen: mat}\&\eqref{gen: isom} and Proposition \ref{prop: barses}.
The only part to explain is which of the two cases in Proposition \ref{prop: barses}\eqref{barses2} arise for $\mathrm{char} \, k = 2$ and either $d=m+n$ with $M=0_{m}\oplus D$ where $D \in k^{n\times n}$ is a nonsingular diagonal matrix, or $d=m+2n$ with $M=0_{m}\oplus \Pi_{2n}$. 

First suppose that $M=0_{m}\oplus D$.
If $X \in \gl_n$ satisfies $X^tM+XM=\lambda M$ then in particular
	\begin{align*}
	\left(X^{t}M+MX\right)_{m+1,m+1} & = \lambda M_{m+1,m+1} \\
	\Leftrightarrow \;\; 2X_{m+1,m+1}D_{1,1} & = \lambda D_{1,1}
	\end{align*}
and therefore $\lambda = 0$ because $\mathrm{char} \, k = 2$ and $D_{1,1} \ne 0$.
It follows that $\omoM = \moM$. 

Now suppose that $M=0_{m} \oplus \Pi_{2n}$.
It is straightforward to check that the block diagonal matrix $X=0_{m+n}\oplus \bI_{n}$ satisfies $X^{t}M+MX=M$ and therefore $\omoM = \mo\left(M\right) \rtimes kX$.
\end{proof}
%
\subsection{Proof of Theorem \ref{thm: der}}
\label{sec: proofpropder}
%
\begin{proof}
Denote the chosen basis of $A_{\left(1\right)}$ by $e_1,\dots,e_d$ and that of $A_{\left(2\right)}$ by $e$.
It follows from the grading on $A$ that a linear map $\p:A\to A$ is a derivation if and only if
	\begin{align}
	\label{eqn: der}
	\p(e_i)e_j + e_i\p(e_j) & = \p(e_i e_j) \\
	\label{eqn: der2}
	\p(e_{i})e + e_{i}\p(e) & = 0
	\end{align}
for all $1\le i,j \le d$.
Moreover, any derivation is determined by its action on $A_{\left(1\right)}$, and annihilates $1_{A}$ and therefore acts trivially on $A_{\left(0\right)}$.

\begin{enumerate}[wide, labelwidth=!, labelindent=10pt]
\item Using the basis of $A_{(1)}$ and $A_{(2)}$ we can identify the Lie algebra of degree $0$ graded linear endomorphisms of the graded vector space $A_{\left(1\right)} \oplus A_{\left(2\right)}$ with $\gl_{d} \oplus \gl_{1}$.
If $(X,\lambda) \in \gl_{d} \oplus \gl_{1}$ then equation \eqref{eqn: der} is equivalent to
	\begin{align*}
	\left(\sum_{k=1}^{d}X_{ki}e_{k}\right)e_{j} + e_{i} \left(\sum_{k=1}^{d}X_{kj}e_{k}\right) & = \p \left(M_{ij}e\right) \\
	\Leftrightarrow \;\;\; \sum_{k=1}^{d}X_{ki}M_{kj} e + \sum_{k=1}^{d}X_{kj}M_{ik} e & = \lambda M_{ij}e  
	\end{align*}
which holds for all $1\le i,j \le d$ if and only if $X^{t}M + MX  = \lambda M$.
Equation \eqref{eqn: der2} holds automatically because both terms on the left hand side are contained in $A_{\left(3\right)}=0$.
As $A$ is generated by $A_{\left(1\right)}$ (or equivalently because $M\ne 0$), the scalar $\lambda$ is uniquely determined by $X$ and therefore $\g \cong \omoM$.
\item In terms of the notation used in the proof of (a), the representation $\pi:\g_{\left(0\right)} \to \gl\left(A_{\left(2\right)}\right)$ is given by $\left(X,\lambda\right) \mapsto \lambda \in \gl_{1}$.
It follows that $\mathrm{Ker}\;\pi = \moM$ as claimed.
\item The grading on $\g$ is such that $\left[\g_{\left(1\right)},\g_{\left(1\right)}\right] \subseteq \g_{\left(2\right)} = 0$ and therefore $\g_{\left(1\right)}$ is abelian.
\item A degree $1$ graded linear endomorphism of the graded vector space $A$ automatically satisfies \eqref{eqn: der} resp.\ \eqref{eqn: der2} because both the left and right hand sides are contained in $A_{\left(3\right)} = 0$ resp.\ $A_{\left(4\right)}=0$.
This shows that $\g_{\left(1\right)} \cong \mathrm{Hom}\left(A_{\left(1\right)},A_{\left(2\right)}\right)$ as vector spaces and $\g_{\left(0\right)}$ representations (see the discussion proceeding Theorem \ref{thm: der} in \S \ref{sec: mainder} for the description of this representation.)
\item In this case $d=1$ so we can choose a basis $e_{1}$ of $A_{\left(1\right)}$ and $e$ of $A_{\left(2\right)}$ such that $e_{1}^{2}=e$.
There is then an isomorphism $k\left[t\right]/\left(t^{3}\right) \cong A$ determined by $t+\left(t^{3}\right) \mapsto e_{1}$.
The remaining statements follow from Remark \ref{rem: Witt} and the fact that $W(1)_{\left(-1\right)}$ is spanned by the partial derivative $\p_{t}$ and so has dimension $1 = \mathrm{dim}\;\mathrm{Hom}\left(A_{\left(1\right)},A_{\left(0\right)}\right)$.
The isomorphism $W(1) \cong \s_{2}$ is elementary \cite{StradeBookI}.
\item By assumption $d=2$ and $M$ is congruent to $\Pi_{2}$ so we can choose the basis $e_{1},e_{2}$ of $A_{\left(1\right)}$ and $e$ of $A_{\left(2\right)}$ such that $e_{1}^{2}=e_{2}^{2}=0$ and $e_{1}e_{2}=e$.
There is then an isomorphism $k\left[t_{1},t_{2}\right]/\left(t_{1}^{2},t_{2}^{2}\right) \cong A$ determined by $t_{1}+\left(t_{1}^{2},t_{2}^{2}\right) \mapsto e_{1}$ and $t_{2}+\left(t_{1}^{2},t_{2}^{2}\right) \mapsto e_{2}$.
The remaining statements follow from Remark \ref{rem: Witt} and the fact that $W(2)_{\left(-1\right)}$ is spanned by the partial derivatives $\p_{t_{1}}$ and $\p_{t_{2}}$ and so has dimension $2 = \mathrm{dim}\;\mathrm{Hom}\left(A_{\left(1\right)},A_{\left(0\right)}\right)$.
\item If statement \eqref{der: other}\ref{other: g(-1)} that $\g_{\left(-1\right)}=0$ holds then statements \eqref{der: other}\ref{other: isom}-\ref{other: iterated} follow:
\begin{enumerate}[label=(\roman*)]
\setcounter{enumii}{1}
	\item The isomorphism $\g \cong \g_{\left(0\right)} \ltimes \mathrm{Hom}\left(A_{\left(1\right)},A_{\left(2\right)}\right)$ follows from the grading on $\g$ and statements \eqref{der: g1ab} and \eqref{der: g0g1}.
	\item The isomorphism $\rad\left(\g\right) \cong \rad\left(\omoM\right) \ltimes \mathrm{Hom}\left(A_{\left(1\right)},A_{\left(2\right)}\right)$ follows from statement \eqref{der: g0} and Proposition \ref{prop: sd}\eqref{sd: rad}.
	\item The isomorphism $\g^{\ssimp} \cong \omoM^{\ssimp}$ follows from statement \eqref{der: g0} and Proposition \ref{prop: sd}\eqref{sd: ss}.
	\item The iterated semi-direct products follow from statement \eqref{der: other}\ref{other: isom}, statement \eqref{der: g0}, and Theorem \ref{thm: table}.
\end{enumerate}
It remains to show that the assumptions in statement \eqref{der: other} imply that $\g_{\left(-1\right)}=0$.

We first establish some general facts about $\g_{\left(-1\right)}$.
We can identify the vector space $\mathrm{Hom}\left(A_{\left(1\right)},A_{\left(0\right)}\right) \oplus \mathrm{Hom}\left(A_{\left(2\right)},A_{\left(1\right)}\right)$ of degree $-1$ graded linear endomorphisms of the graded vector space $A$ with the vector space $k^{d}\oplus k^{d}$, where a pair $\left(v,w\right)$ corresponds to the pair of linear maps $e_{i} \mapsto v_{i}1_{A}$ and $e\mapsto \sum_{i=1}^{d}w_{i}e_{i}$.
If $\p=\left(v,w\right)$ is such a pair then \eqref{eqn: der} is equivalent to
	\begin{align}
	\nonumber
	\left(v_{i}1_{A}\right)e_{j} + e_{i}\left(v_{j}1_{A}\right) & = \chi\left(M_{ij}e\right) \\
	\label{eqn: der-11} \Leftrightarrow \;\;\; v_{i}e_{j} + v_{j}e_{i} & = \sum_{l=1}^{d}M_{ij}w_{l}e_{l}
	\end{align}
and \eqref{eqn: der2} is equivalent to
	\begin{align}
	\nonumber \left(v_{i}1_{A}\right)e + e_{i}\left(\sum_{l}^{d}w_{l}e_{l}\right) & = 0 \\
	\nonumber \Leftrightarrow \;\;\; v_{i}e + \sum_{l=1}^{d}M_{il}w_{l}e & = 0 \\
	\label{eqn: der-12} \Leftrightarrow \;\;\;  v_{i} + \sum_{l=1}^{d}M_{il}w_{l} & = 0.
	\end{align}
If $\p$ is a derivation then it is determined by its action on $A_{\left(1\right)}$ and therefore by $v$.
In particular, to show that $\g_{\left(-1\right)}=0$ it is sufficient to show that the equations \eqref{eqn: der-11} and \eqref{eqn: der-12} have no non-zero common solutions.

Returning to the assumptions in the statement, as $A$ is either commutative or graded commutative the matrix $M$ is either symmetric or antisymmetric.
By changing the basis of $A$ if necessary we can assume that either $d=m+n$ and $M = 0_{m}\oplus D$, or $d=m+2n$ and $M=0_{m}\oplus \Pi_{2n}$.
We deal with the possible cases in turn. 

Suppose that $\mathrm{char}\,k\ne 3$, $d=1$ and $M=D$.
Then $D \in k$, $v,w \in k$, and equations \eqref{eqn: der-11} and \eqref{eqn: der-12} are equivalent to the pair of equations
	\begin{align}
	\label{eqn: derI}
	\begin{cases}
	2v = Dw \\
	v + Dw = 0.
	\end{cases}
	\end{align}
which have no non-zero solutions. 

Suppose that $\mathrm{char}\,k\ne 2$, $d=2$ and $M=\Pi_{2}$.
Then $M_{11}=M_{22}=0$ and so the $i=j=1$ and $i=j=2$ cases of equation \eqref{eqn: der-11} are 
	\begin{align}
	\label{eqn: derP1} 
	\begin{cases} 
	2v_{1}e_{1} = 0 \\
	2v_{2}e_{2} = 0 
	\end{cases}
	\end{align}
which have no non-zero solutions.

In each of the remaining cases either $d=2$ and $M \ne \Pi_{2}$, or $d\ge 3$.
Then for each $1\le i \le d$ we can choose $1\le j\le d$ such that $j\ne i$ and $M_{ij}=0$.
In this case \eqref{eqn: der-11} is
	\[
	v_{i}e_{j} + v_{j}e_{i} = 0 
	\]
which implies that $v_{i}=0$ by the linear independence of $e_i$ and $e_j$.
\qedhere
\end{enumerate}
\end{proof}

\section{The structure of $\moM$ for $M$ symmetric or antisymmetric}
\label{sec: struc}

In this section we compute the radicals and semisimple quotients of the Lie algebras $\moM$ when $M$ is symmetric or antisymmetric.
It follows from the discussion in \S \ref{sec: set-up} that it is sufficient to do this for the cases $d=m+n$ with $M=0_{m}\oplus D$ and $D \in k^{n\times n}$ nonsingular and diagonal, and $d=m+2n$ with $M=0_{m}\oplus \Pi_{2n}$.
In each case we also determine for which $m$ and $n$ is $\moM$ solvable, and in the non-solvable cases describe the derived series.
Throughout \S \ref{sec: struc} the symbols $m$ and $n$ denote \emph{positive} integers (since otherwise either $M=0$ or $M$ is nondegenerate).

The proofs of most of the statements in the Propositions \ref{prop: OI2}-\ref{prop: OP} follow immediately from Theorem \ref{thm: gen} and the facts about $\gl_{m}$, $\s_{m}$, $\so_{n}$, $\moD$ and $\spl_{2n}$ collected in Appendix \ref{sec: bg}.
We give proofs of the statements requiring a further argument, which is in most cases an application of Proposition \ref{prop: gV} below.
We note that the formulas for $\rad\left(\moM\right)$, $\moM^{\ssimp}$ and $\moM^{\left(1\right)}$ stabilise for large enough $m$ and $n$ but differ for small values where $\gl_{m}$, $\so_{n}$ and $\moD$ may be solvable.
\begin{prop}
\label{prop: gV}
In any of the following situations it holds that $\g \bV_{\g}= \bV_{\g}$:
\begin{enumerate}
\item \label{gV: I} $\mathrm{char}\,k=2$, $\g=\mo(D)^{\left(1\right)}$ and $n\ge 3$ where $D \in k^{n\times n}$ is nonsingular and diagonal.
\item \label{gV: so} $\mathrm{char}\,k\ne 2$, $\g=\mo(D)$ and $n\ge 3$ where $D\in k^{n\times n}$ is nonsingular and diagonal. 
\item \label{gV: spl2} $\mathrm{char}\,k=2$, $\g=\spl_{2n}^{\left(2\right)}$ and $n\ge 3$.
\item \label{gV: spl} $\mathrm{char}\,k\ne2$ and $\g=\spl_{2n}$.
\end{enumerate}
\end{prop}
\begin{proof}
\begin{enumerate}[wide, labelwidth=!, labelindent=0pt]
\item It is sufficient to pass to the algebraic closure of $k$ so we can assume that $D=\bI_{n}$ by \S \ref{sec: char2oD}.
Let $E_{ij}$ resp.\ $e_{i}$ denote the standard basis vectors of $\gl_{m}$ resp.\ $k^{m}$. For $2\le j \le m$ we have $\left(E_{1j}+E_{j1}\right) \in \moM^{\left(1\right)}$ and $\left(E_{1j}+E_{j1}\right)e_{1} = e_{j}$. The statement follows.
\item The proof is the same as that of (a).
\item Let $E_{ij}$ resp.\ $e_{l}$ denote the standard basis vectors of $\gl_{n}$ resp.\ $k^{2n}$.
Using the description of the elements of $\spl_{2n}$ in \ref{sec: char2spl}, for $1\le i \le n$ let $X_{i}$ denote the matrix of the form \eqref{eqn: spnmat} with $A=E_{ii}$ and $B=C=0$ .
Then $X_{i} \in \spl_{2n}^{\left(2\right)}$, $X_{i}e_{i}=e_{i}$ and $X_{i}e_{i+n}=e_{i+n}$. The statement follows.
\item In this case $\bV_{\g}$ is simple \cite{HumphreysBook72}, but the statement can also be checked directly by the same argument as in \eqref{gV: spl2} replacing the lower right $n\times n$ block in $X$ by $-E_{ii}$.
\qedhere
\end{enumerate}
\end{proof} 
\begin{prop}
\label{prop: OI2}
Suppose that $\mathrm{char} \, k = 2$, $d=m+n$ and $M=0_{m}\oplus D$, where $D \in k^{n\times n}$ is nonsingular and diagonal.
\begin{enumerate}
\item 
	\begin{align*}
	\rad\left(\moM\right) & \cong \begin{cases}
	\moM   &   \text{if} \; m,n \le 2\\
	\left(\gl_{1} \oplus \moD \right) \ltimes \mathrm{Hom}\left(\bV_{\moD},k_{1}\right)  & \text{if} \; m\ge3,n\le 2 \\
	\left(\gl_{m} \oplus \gl_{1} \right) \ltimes \mathrm{Hom}\left(k_{1}^{\oplus n},\bV_{\gl_{m}}\right) & \text{if} \; m\le2, n\ge 3 \\
	\left(\gl_{1} \oplus \gl_{1} \right) \ltimes \mathrm{Hom}\left(k_{1}^{\oplus n},k_{1}^{\oplus m}\right) & \text{if} \; m\ge 3, n\ge 3. \\
	\end{cases}
	\end{align*}
\item $\moM^{\ssimp} \cong \gl_{m}^{\ssimp} \oplus \moD^{\ssimp}$.
In particular, if $m,n \ge 3$ then
	\begin{align*} 
	\moM^{\ssimp} \cong \begin{cases}
	\pgl_{m} \oplus \moD/k\bI_{n} & \;\text{ if }\; 2|m \\
	\s_{m} \oplus \moD/k\bI_{n} & \;\text{ if }\; 2\nmid m.
	\end{cases}
	\end{align*}
\item If either $m\ge 3$ or $n\ge 3$ then for $i \ge 1$
	\[
	\moM^{\left(i\right)} \cong \left(\s_{m}^{\left(i-1\right)} \oplus \moD^{\left(i\right)} \right) \ltimes \mathrm{Hom}\left(\bV_{\moD^{\left(i\right)}},\bV_{\s_{m}^{\left(i-1\right)}}\right).
	\]
In particular, if $m,n\ge 3$ then
	\begin{align*}
	\moM^{\left(1\right)} & \cong \left(\s_{m} \oplus \moD^{\left(1\right)} \right) \ltimes \mathrm{Hom}\left(\bV_{\moD^{\left(1\right)}},\bV_{\s_{m}}\right) \\
	\moM^{\left(2\right)} & = \moM^{\left(1\right)}.
	\end{align*}
\item $\moM$ is solvable if and only if $m,n \le 2$.
\end{enumerate}
\end{prop}
\begin{proof}
(a),(b) and (d) follow from Theorem \ref{thm: gen}\eqref{gen: rad} and \eqref{gen: ss}, the radical, semisimple quotient and derives series of $\gl_{m}$ \S \ref{sec: char2glsl} and $\moD$ \S \ref{sec: char2oD}, and the fact that the restriction of the vector representation of $\gl_{m}$ to $k\bI_{m} \cong \gl_{1}$ is isomorphic to $k_{1}^{\oplus m}$.

(c) The first statement follows from Proposition \ref{prop: gV}\eqref{gV: I} and Theorem \ref{thm: gen}\eqref{gen: der2}.
The second statement follows from Theorem \ref{thm: gen}\eqref{gen: der2} and the fact that if $m,n\ge 3$ then $\s_{m}$ and $\moD^{\left(1\right)}$ are perfect \S \ref{sec: char2glsl}, \S \ref{sec: char2oD}.
\end{proof}
\begin{prop}
\label{prop: OI}
Suppose that $\mathrm{char} \, k \ne 2$, $d=m+n$ and $M=0_{m}\oplus D$ where $D \in k^{n\times n}$ is nonsingular and diagonal.
\begin{enumerate}
\item  
	\begin{align*}
	\rad\left(\moM\right) & \cong 
	\begin{cases}
	\gl_{1} \ltimes k_{1}^{\oplus mn} & \text{if} \; n=1 \; \text{or} \; n\ge 3 \\
	\left( \gl_{1} \oplus \moD \right) \ltimes \mathrm{Hom}\left(\bV_{\moD} , k_{1}\right) & \text{if} \; n=2.
	\end{cases}
	\end{align*}
\item $\moM^{\ssimp} \cong \gl_{m}^{\ssimp} \oplus \moD^{\ssimp}$.
In particular, if $m\ge 2$ and $n\ge 3$ then
	\begin{align*}
	\moM^{\ssimp} \cong \begin{cases}
	\pgl_{m} \oplus \moD & \;\text{ if }\; \mathrm{char}\,k |m \\
	\s_{m} \oplus \moD & \;\text{ if }\; \mathrm{char}\,k\nmid m.
	\end{cases}
	\end{align*}
\item If $m\ge 2$ or $n\ge 3$ then for $i \ge 1$
	\[
	\moM^{\left(i\right)} \cong \left(\s_{m}^{\left(i-1\right)} \oplus \moD^{\left(i\right)}\right) \ltimes \mathrm{Hom}\left(\bV_{\moD^{\left(i\right)}},\bV_{\s_{m}^{\left(i-1\right)}}\right).
	\]
In particular, if $m\ge 2$ and $n\ge 3$ then
	\begin{align*}
	\moM^{\left(1\right)} & \cong \left(\s_{m} \oplus \moD\right) \ltimes \mathrm{Hom}\left(\bV_{\moD},\bV_{\s_{m}}\right) \\
	\moM^{\left(2\right)} & = \moM^{\left(1\right)}.
	\end{align*}
\item $\moM$ is solvable if and only if $m\le 1$ and $n\le 2$.
\end{enumerate}
\end{prop}
\begin{proof}
(b) and (d) follow from Theorem \ref{thm: gen}\eqref{gen: ss} and the radical, semisimple quotient and derived series of $\gl_{m}$ \S \ref{sec: charn2glsl} and $\moD$ \S \ref{sec: charn2oD}.

(a) If $n=1$ or $n \ge 3$ then $\rad\left(\moD\right)=0$ \S \ref{sec: char2oD} and the restriction of $\mathrm{Hom}\left(\bV_{\moD},\bV_{\gl_{m}}\right)$ to $k\bI_{m} \cong \gl_{1}$ is isomorphic to $k_{1}^{\oplus mn}$.
If $n=2$ then $\moD$ is solvable \S \ref{sec: charn2oD}.
In either case the statement then follows from Theorem \ref{thm: gen}\eqref{gen: rad}.

(c) The first statement follows from Proposition \ref{prop: gV}\eqref{gV: I} and Theorem \ref{thm: gen}\eqref{gen: der2}.
The second statement follows from Theorem \ref{thm: gen}\eqref{gen: der2} and the fact that if $m\ge 2$ and $n\ge 3$ then $\s_{m}$ and $\moD$ are perfect \S \ref{sec: charn2glsl}, \S \ref{sec: charn2oD}.
\end{proof}
\begin{prop}
\label{prop: OP2}
Suppose that $\mathrm{char} \, k = 2$, $d=m+2n$ and $M=0_{m}\oplus \Pi_{2n}$.
\begin{enumerate}
\item  
	\begin{align*}
	\rad\left(\moM\right) & \cong \begin{cases}
	\moM & \text{if}\; m,n \le 2 \\
	\left(\gl_{m} \oplus \gl_{1}\right) \ltimes \mathrm{Hom}\left(k_{1}^{\oplus n},\bV_{\gl_{m}}\right) 
	& \text{if} \; m\le 2, n\ge 3 \\
	\left(\gl_{1}\oplus \spl_{2n}\right) \ltimes \mathrm{Hom}\left(\bV_{\spl_{2n}},k_{1}^{\oplus m}\right) 
	& \text{if} \; m\ge 3, n\le 2 \\
	\left(\gl_{1}\oplus\gl_{1}\right) \ltimes \mathrm{Hom}\left(k_{1}^{\oplus m},k_{1}^{\oplus m}\right) 
	& \text{if} \; m,n \ge 3. \\
	\end{cases}
	\end{align*}
\item $\moM^{\ssimp} \cong \gl_{m}^{\ssimp} \oplus \spl_{2n}^{\ssimp}$.
In particular, if $m,n\ge 3$ then
	\begin{align*}
	\moM^{\ssimp} & \cong \begin{cases}
	\pgl_{m} \oplus \spl_{2n}/k\bI_{2n} & \;\text{ if }\; 2|m \\
	\s_{m} \oplus \spl_{2n}/k\bI_{2n} & \;\text{ if }\; 2\nmid m.
	\end{cases}
	\end{align*}
\item If $m\ge 3$ or $n\ge 3$ then for $i \ge 1$
	\[
	\moM^{\left(i\right)} \cong \left(\s_{m}^{\left(i-1\right)} \oplus \spl_{2n}^{\left(i\right)}\right) \ltimes \mathrm{Hom}\left(\bV_{\spl_{2n}^{\left(i\right)}},\bV_{\s_{m}^{\left(i-1\right)}}\right).
	\]
In particular, if $m,n\ge 3$ then
	\begin{align*}
	\moM^{\left(1\right)} & \cong \left(\s_{m} \oplus \spl_{2n}^{\left(1\right)}\right) \ltimes \mathrm{Hom}\left(\bV_{\spl_{2n}^{\left(1\right)}},\bV_{\s_{m}}\right) \\
	\moM^{\left(2\right)} & \cong \left(\s_{m} \oplus \spl_{2n}^{\left(2\right)}\right) \ltimes \mathrm{Hom}\left(\bV_{\spl_{2n}^{\left(2\right)}},\bV_{\s_{m}}\right) \\
	\moM^{\left(3\right)} & = \moM^{\left(2\right)}.
	\end{align*}
\item $\moM$ is solvable if and only if $m,n \le 2$.
\end{enumerate}
\end{prop}
\begin{proof}
(a),(b) and (d) follow from Theorem \ref{thm: gen}\eqref{gen: rad} and \eqref{gen: ss} and the radical, semisimple quotient and derived series of $\gl_{m}$ \S \ref{sec: char2glsl} and $\spl_{2n}$ \S \ref{sec: char2spl}.

\item (c) The first statement follows from Proposition \ref{prop: gV}\eqref{gV: I} and Theorem \ref{thm: gen}\eqref{gen: der2}.
The second statement follows from Theorem \ref{thm: gen}\eqref{gen: der2} and the fact that if $m,n\ge 3$ then $\s_{m}$ and $\spl_{2n}^{\left(2\right)}$ are perfect \S \ref{sec: char2glsl}, \S \ref{sec: char2spl}.
\end{proof}
\begin{prop}
\label{prop: OP}
Suppose that $\mathrm{char} \, k \ne 2$, $d=m+n$ and $M=0_{m}\oplus \Pi_{2n}$.
\begin{enumerate}
\item $\rad\left(\moM\right)  \cong \gl_{1} \ltimes k_{1}^{\oplus mn}$.
\item $\moM^{\ssimp} \cong \gl_{m}^{\ssimp} \oplus \spl_{2n}$.
In particular, if $m\ge2$ then
	\begin{align*}
	\moM^{\ssimp} & \cong \begin{cases}
	\pgl_{m} \oplus \spl_{2n} & \;\text{ if }\; \mathrm{char}\,k|m \\
	\s_{m} \oplus \spl_{2n} & \;\text{ if }\; \mathrm{char}\,k\nmid m.
	\end{cases}
	\end{align*}
\item If $m=1$ then 
	\begin{align*}
	\moM^{\left(1\right)} & \cong \spl_{2n} \ltimes \bV_{\spl_{2n}} \\
	\moM^{\left(2\right)} & = \moM^{\left(1\right)}.
	\end{align*}
If $m\ge 2$ then 
	\begin{align*}
	\moM^{\left(1\right)} & \cong \left(\s_{m} \oplus \spl_{2n}\right) \ltimes \mathrm{Hom}\left(\bV_{\spl_{2n}},\bV_{\s_{m}}\right) \\
	\moM^{\left(2\right)} & = \moM^{\left(1\right)}.
	\end{align*}
\item $\moM$ is not solvable.
\end{enumerate}
\end{prop}
\begin{proof}
(b) and (d) follow from Theorem \ref{thm: gen}\eqref{gen: ss} and the radical and semisimple quotient of $\gl_{m}$ \S \ref{sec: charn2glsl} and $\spl_{2n}$ \S \ref{sec: charn2spl}.
\item (a) For any $n,m\ge 1$ it holds that $\rad\left(\spl_{2n}\right)=0$ \S \ref{sec: charn2spl} and $\rad(\gl_{m}) = k\bI_{m}$ \S \ref{sec: charn2glsl}, and the restriction of $\mathrm{Hom}\left(\bV_{\spl_{2n}},\bV_{\gl_{m}}\right)$ to $k\bI_{m} \cong \gl_{1}$ is isomorphic to $k_{1}^{\oplus mn}$.
The statement then follows from Theorem \ref{thm: gen}\eqref{gen: rad}.

\item (c) The first statement follows from Proposition \ref{prop: gV}\eqref{gV: I} and Theorem \ref{thm: gen}\eqref{gen: der2}.
The second statement follows from Theorem \ref{thm: gen}\eqref{gen: der2} and the fact that if $m\ge 2$ and $n\ge 3$ then $\s_{m}$ and $\spl_{2n}$ are perfect \S \ref{sec: charn2glsl}, \S \ref{sec: charn2spl}.
\end{proof}


\begin{appendix} \section{The Lie algebras $\gl_{d}$, $\s_{d}$, $\spl_{2n}$, $\so_{n}$ and $\mo(D)$} \label{sec: bg}
In this Appendix we collect a number of statements on the structure of the classical Lie algebras $\gl_{n}$ and $\s_{n}$, and the Lie algebras $\spl_{2n}$, $\so_{n}$ and $\moD$ for $D \in k^{n \times n}$ a nonsingular diagonal matrix. 
These results are well known for $\mathrm{char}\,k \ne 2$, but perhaps less well known in the $\mathrm{char}\,k = 2$ situation.
We state the relevant results from the following references below:
\begin{itemize}
\item The structure of $\gl_{n}$ and $\s_{n}$, and for $\mathrm{char}\,k \ne 2$ that of $\so_{n}$ and $\spl_{2n}$ is well known \cite{HumphreysBook72,PremetS06}.
\item Hogeweij \cite{Hogeweij82} determined the structure of the Chevalley algebras over a field of arbitrary characteristic $p>0$.
These Lie algebras are mostly simple, except for $\s_{n}$ when $p|n$, and some exceptions for low $p$ including $\spl_{2n}$ when $p=2$, which corresponds to the `$\mathrm{C}_{l}$, $p=2$, universal type' entry in \cite[Table 1]{Hogeweij82}. 
The structure of $\spl_{2n}$ is also determined in \cite[\S 12]{ChaktouraS14}.
\item The structure of $\moD$ is determined by Chaktoura and Szechtman \cite[\S 8, 9, 11]{ChaktouraS14}.
Note that in loc.\ cit.\ the Lie algebra which we denote by $\moI$ is denoted by $\so_{n}$.
\end{itemize}
The Appendix ends with Remarks \ref{rem: schemes} and \ref{rem: chev} in which, for the case $\mathrm{char}\, k = 2$, we explain the relation  between the Lie algebras $\moI$, $\mo_{n}$ and $\so_{n}$, the related group schemes, and the Chevalley algebras.
In particular these remarks explain why we do not denote $\moI$ by $\mo_{n}$ or $\so_{n}$.
%
\begin{remi}
There are some errors in the literature regarding the structure of the Lie algebras $\moM$.
The second Lemma in the preprint \cite[\S 6.4]{Lebedev06} incorrectly states that $\mo\left(\bI_{4}\right)^{\left(1\right)}$ is simple.
As stated in \cite[Note 7.9]{ChaktouraS14} there are two errors in Bourbaki \cite[Chapter I, \S6, Exercise 25(b)]{BourbakiBook89} regarding the ideal structure of $\spl_{2n}$ in characteristic $2$.
The statement in loc.\ cit.\ that the only non-zero proper ideals of $\spl_{2n}^{\left(1\right)}$ are $\spl_{2n}^{\left(2\right)}$ and $k\bI_{2n}$ is correct.
\end{remi}
\subsection{The $\mathrm{char}\,k \ne 2$ case}
\label{sec: charn2}
Throughout \S \ref{sec: charn2} we assume that $\mathrm{char}\, k \ne 2$.
\subsubsection{$\gl_{n}$, $\s_{n}$, $n\ge 2$ \cite[\S 3.1]{PremetS06} }
\label{sec: charn2glsl}
The derived subalgebra $\gl_{n}^{\left(1\right)} = \s_{n}$ and $\s_{n}$ is perfect.
The radical of $\gl_{n}$ is $\rad\left(\gl_{n}\right) = k\bI_{n}$.
The semisimple quotient $\pgl_{n} := \gl_{n} / k\bI_{n}$ is isomorphic to $\s_{n}$ if and only if $\mathrm{char}\,k \nmid m$. 
The Lie algebra $\s_{n}$ is simple if $\mathrm{char}\,k\nmid m$, otherwise $\rad\left(\s_{n}\right) = k\bI_{n}$ and $\psl_{n} := \s_{n} / k\bI_{n}$ is simple.
\subsubsection{$\mo(D)$, $n\ge 1$ \cite[\S 7,8,9]{ChaktouraS14} }
\label{sec: charn2oD}
Let $D \in k^{n\times n}$ be a nonsingular diagonal matrix.
The Lie algebra $\mo(D)$ consists of the $n\times n$ matrices $X$ satisfying $D_{ii}X_{ij}+D_{jj}X_{ji}=0$ for all $1\le i,j\le n$.
The Lie algebra $\mo(D)$ is zero if $n=1$, one dimensional and abelian if $n = 2$, and simple if $n=3$ or $n \ge 5$.
If $n=4$ then $\mo(D)$ is simple if $\mathrm{det} \, D \notin k^{\times 2}$, and the direct sum of two simple ideals if $\mathrm{de} \, D \in k^{\times 2}$.
Note that if every element in $k$ is a square then $D$ is congruent to $\bI_{n}$ and $\mo(D)$ is isomorphic to $\moI$.
\subsubsection{$\spl_{2n}$, $n\ge 1$ \cite[\S I]{HumphreysBook72} }
\label{sec: charn2spl}
The symplectic Lie algebra $\spl_{2n}$ is simple for any $n \ge 1$.
\subsection{The $\mathrm{char}\,k = 2$ case}
\label{sec: char2}
Throughout \S \ref{sec: char2} we assume that $\mathrm{char}\,k = 2$. 
\subsubsection{$\gl_{n}$, $\s_{n}$, $n\ge 2$ \cite[\S 3.1]{PremetS06} }
\label{sec: char2glsl}
The derived subalgebra $\gl_{n}^{\left(1\right)} = \s_{n}$.
For $n=2$ the Lie algebras $\gl_{n}$ and $\s_{n}$ are solvable.
Assume then that $n \ge 3$.
The radical of $\gl_{n}$ is $\rad\left(\gl_{n}\right) = k\bI_{n}$ and the semisimple quotient is $\pgl_{n} := \gl_{n} / k\bI_{n}$.
The Lie algebra $\pgl_{n}$ is isomorphic to $\s_{n}$ if and only if $2\nmid n$. 
The Lie algebra $\s_{n}$ is simple if and only if $2 \nmid  n$, otherwise $\rad\left(\s_{n}\right) = k\bI_{n}$ and the semisimple quotient $\psl_{n} := \s_{n} / k\bI_{n}$ is simple.
\subsubsection{$\mo(D)$ \cite[\S 7,8,11]{ChaktouraS14}}
\label{sec: char2oD}
Suppose that $D \in k^{n\times n}$ is a nonsingular diagonal matrix.
The Lie algebra $\mo(D)$ consists of the $n\times n$ matrices $X$ satisfying $D_{ii}X_{ij}+D_{jj}X_{ji}=0$ for all $1\le i,j\le n$.
The derived subalgebra $\mo(D)^{\left(1\right)}$ consists of those elements of $\mo(D)$ with zeroes on the diagonal.
The Lie algebra $\mo(D)^{\left(1\right)}$ is solvable if $n\le 2$ and simple if $n=3$ or $n\ge 5$.
If $n=4$ and $\mathrm{der} \, D$ is not a square in $k$ then $\mo(D)$ is simple.
If $n=4$ and $\mathrm{det} \, D$ is a square in $k$ (e.g.\ if $k$ is algebraically closed) then $\mo(D)$ is perfect and equal to the semidirect product of a three dimensional simple subalgebra and an abelian ideal.
If $n\ge 3$ then $\rad\left(D\right) = k\bI_{n}$.
Note that if every element in $k$ is a square then $D$ is congruent to $\bI_{n}$ and $\mo(D)$ is isomorphic to $\moI$.
\subsubsection{$\spl_{2n}$ \cite[Table 1]{Hogeweij82},\cite[\S 12]{ChaktouraS14} }
\label{sec: char2spl}
The Lie algebra $\spl_{2n}$ consists of the $2n\times 2n$ matrices of the form
\begin{equation}
	\label{eqn: spnmat}
	\begin{pmatrix}
	A&B\\
	C&A^{t}
	\end{pmatrix}
\end{equation}
with $A\in \gl_{n}$ and $B,C$ antisymmetric.
The derived subalgebra $\spl_{2n}^{\left(1\right)}$ resp.\ $\spl_{2n}^{\left(2\right)}$ consists of all matrices of the form \eqref{eqn: spnmat} with $A\in \gl_{n}$ and $B,C$ alternating (antisymmetric with zeroes on the diagonal), resp.\ $A\in \s_{n}$ and $B,C$ alternating.
For $n \le 2$ the Lie algebra $\spl_{2n}$ is solvable.
For $n\ge 3$ the Lie algebra $\spl_{2n}^{\left(2\right)}$ is perfect, and $\spl_{2n}^{\left(2\right)}$ resp.\ $\spl_{2n}^{\left(2\right)}/k\bI_{n}$ is simple if $n$ is odd resp.\ even.
For $m\ge 3$ the radical of $\spl_{2n}$ is $\rad\left(\spl_{2n}\right) = k\bI_{2n}$.
(In \cite{Hogeweij82} the derived subalgebras of are given in terms of root data as $\spl_{2n}^{\left(1\right)} = \h + \mathfrak{e}_{S}$ and $\spl_{2n}^{\left(2\right)} = \h_{S} + \mathfrak{e}_{S}$, see loc.\ cit.\ for the notation.)
%
\begin{remi}
\label{rem: rad}
The statements about the radicals of $\moI$ and $\spl_{2n}$ are made in \cite[\S 5]{AhmadiIS14} with proofs referenced to \cite{ChaktouraS14}.
As it is not immediately clear to us how these statements follow from the results in \cite{ChaktouraS14} we have provided a proof in the following Proposition. 
\end{remi}
\begin{prop}
\label{prop: rad}
Assume that $\mathrm{char}\,k = 2$, $n\ge 3$, and $D \in k^{n \times n}$ is nonsingular and diagonal.
\begin{enumerate}
\item $\rad(\mo(D)) = Z(\mo(D)) = k\bI_{n}$.
\item $\rad(\spl_{2n}) = Z(\spl_{2n}) = k\bI_{2n}$.
\end{enumerate}
\end{prop}
\begin{proof}
(a) Suppose that $I \subseteq \mo(D)$ is a solvable ideal.
Then $\left[\mo(D),I\right] \subseteq \mo(D)^{\left(1\right)} \cap I$ is a solvable ideal that is contained in $\mo(D)^{\left(1\right)}$.
By \cite[Thm. 1.2]{ChaktouraS14}, if $n=3$ or $n\ge 5$ then $\mo(D)^{\left(1\right)}$ is simple, and by Proposition 11.2 in loc.\ cit.\ $\mo(D)^{\left(1\right)}$ is perfect and irreducible as an $\mo(D)$ module.
It follows that $\mo(D)^{\left(1\right)}$ contains no non-zero solvable ideals of $\mo(D)$ and therefore $\left[\mo(D),I\right] = 0$ and $I\subseteq Z\left(\mo(D)\right)$.
It is straightforward to check that $Z\left(D\right) = k\bI_{n}$.

(b) It follows from the `$\mathrm{C}_{l}$ of universal type entry' in \cite[Table 1]{Hogeweij82} that the only non-zero solvable ideal in $\spl_{2n}$ is $Z(\spl_{2n}) = k\bI_{2n}$.  
The statement can also be deduced from the fact that the ideal $\spl_{2n}^{\left(2\right)}$ has codimension one in $\spl_{2n}^{\left(1\right)}$ and \cite[Thm.\ 12.1]{ChaktouraS14} which shows that $\spl_{2n}^{\left(2\right)}$ is simple if $n$ is odd and contains a single non-zero proper ideal $k\bI_{2n}$ is $n$ if even.
\end{proof}
\begin{remi}{\emph{Group schemes.}}
\label{rem: schemes}
The Lie algebra $\moI$ is isomorphic to the Lie algebra of the group scheme over $k$ with group of $R$-points $\left\{ g \in \mathrm{GL}_{n}\left(R\right) \; | \; g^{t}g = \bI_{n} \right\}$ for $R$ a $k$-algebra (see \cite{DemazureGBook80,MilneBook17} for generalities on group schemes).
This group scheme is in general not isomorphic to a (special) orthogonal group and is not smooth \cite[\S VI Exercise 15]{KnusMRT98},\cite[Exercise 24.2]{MilneBook17}.

The Lie algebra $\moI$ is in general \emph{not} isomorphic to the Lie algebra $\mo_{n}$ of the orthogonal group scheme $\mathrm{O}_{n}$ or to the Lie algebra $\so_{n}$ of the special orthogonal group scheme $\mathrm{SO}_{n}$.
These group schemes are defined using nondegenerate \emph{quadratic forms} rather than bilinear forms \cite[\S 21j \& \S 24i]{MilneBook17}.
In the even case $n=2m$ the group schemes $\mathrm{O}_{2m}$ and $\mathrm{SO}_{2m}$ are smooth, their Lie algebras coincide, and there is a monomorphism $\mathrm{SO}_{2m} \to \mathrm{Sp}_{2m}$ to which the Lie functor associates an isomorphism of $\so_{2m}$ onto the ideal $\spl_{2m}^{(1)}$ in $\spl_{2m}$.
In the odd case $n=2m+1$ the group scheme $\mathrm{O}_{2m+1}$ is non-smooth, $\mathrm{SO}_{2m+1}$ is its reduced subgroup, $\so_{2m+1}$ has codimension one in $\mo_{2m+1}$, and there is a surjective isogeny $\mathrm{SO}_{2m+1} \to \mathrm{Sp}_{2m}$ to which the Lie functor associates a surjective homomorphism from $\so_{2m+1}$ onto $\spl_{2m}^{(1)}$, the kernel of which is abelian of dimension $2m$.
For details see \cite[\S A]{ConradAGINotes},\cite[\S 3,4]{AcharHS11}, and compare the calculation of $\so_{2n}$ in the proof of \cite[Prop.\ A.2.3]{ConradAGINotes} with the description of $\spl_{2n}^{(1)}$ in \S \ref{sec: char2spl} above.
Note that in some older references e.g.\ Borel \cite[\S 23.6]{BorelBook91} and Hesselink \cite{Hesselink79}, $\mathrm{O}_{n}$ denotes the associated `classical' algebraic subgroup of $\mathrm{GL}_{n}(k)$ and corresponds to the reduced subgroup of the orthogonal group scheme, which in the odd case $n=2m+1$ has Lie algebra $\so_{2m+1}$ not $\mo_{2m+1}$.
\end{remi}
\begin{remi}{\emph{Chevalley algebras.}}
\label{rem: chev}
The Lie algebras $\moI$ and $\moI^{(1)}$ are in general \emph{not} isomorphic to any of the Chevalley algebras of type $\mathrm{B}$ or $\mathrm{D}$. The latter are constructed as $\mathrm{mod}\;2$ reductions of integral forms of the simple complex Lie algebras $\so_{2n+1}$ and $\so_{2n}$ \cite[\S 25.4]{HumphreysBook72}.
The Chevalley algebras of these types have a non-zero centre and satisfy $\mathrm{dim}\; \g/\g^{\left(1\right)} \le 2$ \cite{Hogeweij82}, whereas $\moI^{\left(1\right)}$ is simple if $n=3$ or $n\ge 5$ and $\mathrm{dim} \left( \moI/\moI^{\left(1\right)} \right)= n$.
\end{remi}
\end{appendix}

\bibliography{mybib}{}
\bibliographystyle{siam}

\vspace{0.5cm}
\noindent James Waldron, School of Mathematics, Statistics and Physics, Newcastle University, Newcastle upon Tyne, NE1 7RU, UK.
\newline
\noindent Email address: james.waldron@newcastle.ac.uk

\end{document}